\documentclass[12pt,twoside]{article}
\usepackage{amssymb}
\usepackage{amsmath}
\usepackage{amsfonts}
\usepackage{graphicx}
\usepackage{enumerate}
\usepackage{amsthm}
\usepackage{calc}
\usepackage{tikz-cd}
\usepackage{float}
\usepackage{newlfont}
\usetikzlibrary{calc}
\UseRawInputEncoding

\newcommand{\RN}[1]{%
  \textup{\uppercase\expandafter{\romannumeral#1}}%
}
\date{}

\begin{document}

\vspace*{1.5cm}

\centerline{}

\centerline {\Large{\bf SIMPLE DERIVATIONS IN TWO VARIABLES}}
\centerline{}

\centerline{\bf {Pankaj Shukla\footnote{The author is financially supported by the JRF grant from CSIR India, Sr.
No. 09/1022(16075)/2022-EMR-I.} \& Anand Parkash }}

\centerline{Department of Mathematics}

\centerline{Indian Institute of Technology Indore, Indore, India.}

\centerline{E-Mail: phd2201141002@iiti.ac.in; anandparkash@iiti.ac.in}

\centerline{}

\newtheorem{Theorem}{\quad Theorem}[section]

\newtheorem{Corollary}[Theorem]{\quad Corollary}

\newtheorem{Lemma}[Theorem]{\quad Lemma}

\newtheorem{Proposition}[Theorem]{\quad Proposition}

\theoremstyle{definition}

\newtheorem{Definition}[Theorem]{\quad Definition}

\newtheorem{Example}[Theorem]{\quad Example}

\newtheorem{Remark}[Theorem]{\quad Remark}

\numberwithin{equation}{Theorem}

\begin{abstract}
Let $k$ be a field of characteristic zero. If $c_1, c_2\in k\setminus \{0\}, s,t\geq 1$ and $u\geq 0$,
then it is shown that the $k$-derivations $\partial_x + x^u(c_1x^ty^s+c_2)\partial_y$ and
$\partial_x + x^u(c_1x^t+c_2y^{s+1})\partial_y$ of $k[x,y]$ are simple.
We also give a necessary and sufficient condition for the $k$-derivation $y^r\partial_x + (c_1x^{t_1}y^{s_1}+c_2x^{t_2}y^{s_2})\partial_y$,
where $r, t_1, s_1, t_2, s_2 \geq 0$ and $c_1, c_2\in k$,
of $k[x,y]$ to be simple.
\end{abstract}


\section{Introduction}

Let $k$ be a field of characteristic zero and let $A$ be a $k$-algebra. A derivation $d$ of $A$
is said to be simple if $\{0\}$ and $A$ are the only $d$-differential ideals of $A$.
This kind of derivations are significant in several areas such
as constructing nonholonomic irreducible modules over Weyl algebra \cite{c07},
finding examples of noncommutative simple rings,
simple lie rings \cite{g89} and in proving simplicity of Ore extension \cite{g89}.
Some applications and characteristics of simple derivations are also provided in \cite{b03,j81,n94}.
Therefore it is natural and interesting to find simple derivations of commutative rings.
However, finding examples of simple derivations is difficult. It is quite complicated
even for the polynomial ring in two variables over a field of characteristic zero.
The question of $d$-simplicity of a commutative ring has been studied by several authors,
e.g., \cite{c07,g09,k12,k13,k14,l08,m01,n08,y19}.

Let $A=k[x,y]$ and $d$ be a $k$-derivation of $A$. It is obviously hard to find out conditions on
$d(x)$ and $d(y)$ for $d$ to be simple, in complete generality.
In \cite{l08}, Lequain gave an algorithmic characterization of all simple Shamsuddin
derivations (i.e. $d=\partial_x+(a(x)y+b(x))\partial_y$). In \cite{n08}, Nowicki prove that
the derivation $d=\partial_x+(y^s+cx)\partial_y$ is simple if $s\geq 2$ and $c\in k\setminus\{0\}$.
In \cite{g09}, Gavran proved that the derivation $d=\partial_x+(y^s+cx^t)\partial_y$ is simple if $s\geq 2, t\geq 1$ and $c\in k\setminus\{0\}$.
In \cite{k12,k13,k14}, Kour and Maloo proved that the derivation $d=y^r\partial_x+(y^sx^t+c)\partial_y$
is simple if $s,t\geq 1$, $r\geq 0$ and $c\in k\setminus\{0\}$.
In this article, we prove that the $k$-derivation $\partial_x + x^u(c_1x^ty^s+c_2)\partial_y$ is simple
if $s,t\geq 1$, $u\geq 0$ and  $c_1, c_2\in k\setminus \{0\}$; and the $k$-derivation
$\partial_x + x^u(c_1x^t+c_2y^{s})\partial_y$ of $k[x,y]$ is simple if $s\geq 2,t\geq 1, u\geq 0$ and  $c_1, c_2\in k\setminus \{0\}$.
We also give a necessary and sufficient condition for the $k$-derivation $y^r\partial_x + (c_1x^{t_1}y^{s_1}+c_2x^{t_2}y^{s_2})\partial_y$,
where $r, t_1, s_1, t_2, s_2 \geq 0$ and $c_1, c_2\in k$,
of $k[x,y]$ to be simple.

\section{Preliminaries}

Throughout this article, $A$ is a commutative ring with unity and $k$ is a field of characteristic zero.
We now recall some definitions.

A derivation $d$ of $A$ is a map $d:A\rightarrow A$ satisfying $d(a+b)=d(a)+d(b)$ and
$d(ab)=ad(b)+bd(a)$ for all $a,b \in A$.
A derivation $d$ of $A$ is uniquely extendable to every localization of $A$. The
extension of $d$ to a localization of $A$ shall also be denoted by $d$.

An ideal $I$ of $A$ is called $d$-differential if $d(I)\subseteq I$. A derivation $d$ of $A$
is called simple if $\{0\}$ and $A$ are the only $d$-differential ideals of $A$.

An element $f \in A$  is called a Darboux element of $d$ if $f$ is a nonzero nonunit in $A$ and $d(f)=\lambda f$
for some $\lambda \in A$.

Let $A$ be a $k$-algebra. Then by a $k$-derivation of $A$, we mean a derivation of $A$ which is also $k$-linear.
Two $k$-derivations $d$ and $\overline{d}$ are called equivalent if there exists a $k$-algebra automorphism
$\theta$ such that $\overline{d} = \theta d \theta^{-1}$. It is easy to see that if two $k$-derivations $d$ and $\overline{d}$ are equivalent,
then $d$ is simple if and only if $\overline{d}$ is simple.

\begin{Definition}
Let $a,b\in k[x], a\neq 0$. By division algorithm, inductively define $q_1, q_2, \dots \in k[x]$ and
$r_1, r_2, \dots \in k[x]$ such that $\deg_x(r_i) < \deg_x(a)$ for every $i$ and $b=q_1 a+r_1$,
$\partial_x(q_1) = q_2a+r_2,\dots, \partial_x(q_i) = q_{i+1}a+r_{i+1}$, \dots. Put $q_0=1$. Then
there exists an integer $t\geq 0$ such that $q_t\neq 0$ and $q_{t+1}=0$. The polynomial $\sum_{i=1}^{t+1} r_i$
will be denoted by $\mathfrak{p}(a,b)$.
\end{Definition}

We now recall some results.

\begin{Lemma}\label{000}
Let $a,b\in k[x], a\neq 0, \epsilon \in k\setminus \{0\}$ and $m,n\geq 0$. Then:
\begin{enumerate}
\item[{\rm (a)}] $\deg_x(\mathfrak{p}(a,b)) < \deg_x(a)$.

\item[{\rm (b)}] $\mathfrak{p}(a,\epsilon b)= \epsilon \mathfrak{p}(a,b)$.

\item[{\rm (c)}] $\mathfrak{p}(a,b)=b$ if $\deg_x(a) > \deg_x(b)$.

\item[{\rm (d)}] $\mathfrak{p}(x^n, x^{m+n+1})= (m+1)   \mathfrak{p}(x^n, x^m)$.

\item[{\rm (e)}] If $m>n$ then $\mathfrak{p}(x^n, x^m)= 0 \Leftrightarrow (n+1)|(m-n)$.
\end{enumerate}
\end{Lemma}
\begin{proof}
(a), (b) and (c) are well known.

(d) As $x^{m+n+1}=x^{m+1}x^n+0$, therefore $$\mathfrak{p}(x^n, x^{m+n+1})= \mathfrak{p}(x^n, (m+1)x^m)+0= (m+1)   \mathfrak{p}(x^n, x^m).$$

(e) Dividing $m$ by $n+1$, there exists $q,r\geq 0$ such that $m=q(n+1)+r$, where $r\leq n$. Using (c) and (d), we get
$$\mathfrak{p}(x^n, x^m)= 0 \Leftrightarrow \mathfrak{p}(x^n, x^r)= 0 \Leftrightarrow r=n \Leftrightarrow (n+1)|(m-n).$$
\end{proof}

\begin{Theorem}\label{001}
Let $d$ denote the $k$-derivation $\partial_x + (ay+b)\partial_y$ of $k[x,y]$,
where $a,b\in K[x]$ and $a\neq 0$. Then $d$ is a simple derivation of $k[x,y]$ if and only if the polynomial $\mathfrak{p}(a,b) \neq 0$.
\end{Theorem}
\begin{proof}
It follows from \cite[Theorem~2.1 and Lemma~2.3]{l08}.
\end{proof}

\begin{Proposition}\label{p01}
Let $d$ be a $k$-derivation of $k[x,y]$ such that $(d(x),d(y))=k[x,y]$. Then $d$ is a simple derivation of $k[x,y]$
if and only if $d$ has no Darboux elements in $k[x,y]$.
\end{Proposition}
\begin{proof}
It follows from \cite[Proposition~2.1]{k12}.
\end{proof}

\begin{Theorem}\label{002}
Let $d$ denote the $k$-derivation $y^r\partial_x + (c_1x^ty^s+c_2)\partial_y$ of $k[x,y]$,
where $r \geq 0, s,t \geq 1$ and $c_1, c_2\in K\setminus \{0\}$. Then $d$ is a simple derivation of $k[x,y]$.
\end{Theorem}
\begin{proof}
By \cite[Proposition~13.1.3]{n94}, we can assume that $k$
is an algebraically closed field. Let $\bar{d}$ denote the $k$-derivation $y^r\partial_x + (x^ty^s+c_2)\partial_y$ of $k[x,y]$.
Then by \cite[Theorem~4.1 and Theorem~6.1]{k14}, $\bar{d}$ is simple.
Let $\alpha \in k$ be such that $\alpha^{t(r+1)+s}=c_1^{-1}$. Define the $k$-algebra automorphism
$\theta:k[x,y]\rightarrow k[x,y]$ by $\theta(x)=\alpha^{r+1} x$ and $\theta(y)=\alpha y$.
Then $\theta d \theta^{-1} = \alpha^{-1}\bar{d}$ and hence $d$ is simple.
\end{proof}

From here onward we shall assume that $k$ is an algebraically closed field of characteristic $0$.

\section{The derivation $\partial_x + x^u(c_1x^ty^s+c_2)\partial_y$ of $k[x,y]$}

For this section, we fix some notations. Let $x, y$ be indeterminates, $u\geq 0, s,t\geq 1$ be integers.
Put $v=t+(u+1)s$. Let $\omega \in k$ be a primitive $v$th root of unity. Let $\psi$ denote the $k$-algebra automorphism of
$k[x,y]$ given by $\psi(x)=\omega x$ and $\psi(y)=\omega^{u+1} y$. These notations shall remain fixed throughout this section.

In this section, we discuss the simplicity of the $k$-derivation $\partial_x + x^u(c_1x^ty^s+c_2)\partial_y$ of $k[x,y]$.

We now prove some lemmas.

\begin{Lemma}\label{l01}
The $k$-derivation $\bar{d} = \partial_x + x^u(c_1x^ty^s+c_2)\partial_y$, where
$c_1, c_2\in k\setminus \{0\}$, of $k[x,y]$ is simple if and only if the $k$-derivation
$d = \partial_x + x^u(x^ty^s+1)\partial_y$ of $k[x,y]$ is simple.
\end{Lemma}

\begin{proof}
Let $\alpha \in k$ be such that $\alpha^v=c_2^{1-s}c_1^{-1}$. Define the $k$-algebra automorphism
$\theta:k[x,y]\rightarrow k[x,y]$ by $\theta(x)=\alpha x$ and $\theta(y)=c_2\alpha^{u+1} y$.
Then $\theta \bar{d} \theta^{-1} = \alpha^{-1}d$. Therefore, $\bar{d}$ is simple if and only if $d$ is simple.
\end{proof}

\begin{Lemma}\label{l02}
Let $d$ denote the $k$-derivation $\partial_x + x^u(x^ty^s+1)\partial_y$ of $k[x,y]$.
If $d$ is not a simple derivation of $k[x,y]$, then there exists
$F=\sum_{i=0}^m b_ix^i$, where $m>0$ and $b_i\in k[y]$ for $0\leq i \leq m$, such that
$d(F)=bx^{t+u}y^{s-1}F$, $\psi(F)=F$, $b$ is a non negative integer and $b_m=y^b$.
\end{Lemma}
\begin{proof}
By Proposition~\ref{p01}, there exist $f\in k[x,y] \setminus k$ and $\lambda \in k[x,y]$ such that $d(f)=\lambda f$. It is easy
to prove that $f\notin k[y]$. Hence $\deg_x(f) >0$. Note that if $\lambda \neq 0$, then $\deg_x(\lambda)\leq t+u$ and
$\deg_y(\lambda)\leq s-1$.

Let $F=\prod_{i=0}^{v-1} \psi^i(f)$. Then $\psi(F)=F$ and $\deg_x(F)>0$. Note that for every positive integer $i$,
we have $\psi^{-i}d\psi^i=\omega^id$ and therefore $d(F)=\Lambda F$, where $\Lambda=\sum_{i=0}^{v-1}\omega^i\psi^i(\lambda)$.
Now, if $\Lambda \neq 0$, then
$\deg_x(\Lambda)\leq \deg_x(\lambda)\leq t+u$ and $\deg_y(\Lambda)\leq \deg_y(\lambda)\leq s-1$.
As $\psi^{-1}d\psi=\omega d$ and $\psi(F)=F$, therefore $\psi(\Lambda)=\omega^{-1} \Lambda$.
For $0\leq i \leq t+u$ and $0\leq j \leq s-1$, equating the coefficient of $x^iy^j$ in $\psi(\Lambda)=\omega^{-1} \Lambda$, we
get $\Lambda = b x^{t+u}y^{s-1}$ for some $b\in k$.

Write $F= \sum_{i=0}^m b_ix^i$, where $m\geq 1$, $b_i\in k[y]$
for $0\leq i \leq m$ and $b_m\neq 0$. Then by equating the coefficient of $x^{m+t+u}$ in $d(F)=bx^{t+u}y^{s-1}F$,
we get $(b_m)_yy^s=bb_my^{s-1}$. It follows that $b$ is a non negative integer and $b_m= \gamma y^b$ for some $\gamma \in k\setminus \{0\}$.
Without loss of generality we may assume that $\gamma =1$.
\end{proof}

\begin{Lemma}\label{l03}
Let $d$ denote the $k$-derivation $\partial_x + x^u(x^ty^s+1)\partial_y$ of $k[x,y]$.
If $d$ is not a simple derivation of $k[x,y]$, then there exists a Darboux element
$g=\sum_{i=0}^n a_ix^i$, where $a_i\in k[y,1/y]$ for $0\leq i \leq n$, of $d$ in $k[x,y,1/y]$ such that
\begin{enumerate}
\item[{\rm (a)}] $a_n=1, n=lt$ is a positive integer, and $v|n$.

\item[{\rm (b)}] $yd(g) = -ax^ug$, where $a$ is a non negative integer and $v|a$.

\item[{\rm (c)}] $\psi(g) = g$.

\item[{\rm (d)}] $a_{n-t+i} = 0$ for $1\leq i \leq t-1$.

\item[{\rm (e)}] Let $a_{i,j}$ denote the coefficient of $y^j$ in $a_i$ for $0\leq i \leq n$ and $j\in \mathbb{Z}$;
and $a_{i,j}=0$ for $i\in \mathbb{Z}\setminus \{0,1,\dots,n\}$ and $j\in \mathbb{Z}$. Then for every $\alpha \geq 0$, we have
$a_{i,j} = 0$ for every pair $(i,j)$ such that either $i\geq n-\alpha t$ and $j<-\alpha s$ or $n-(\alpha+1)t < i < n-\alpha t$ and $j<-\alpha s+1$.

\item[{\rm (f)}] Let $\alpha_i = a_{n-it, -is}$ for $i\in \mathbb{Z}$. Then $\alpha_i = 0$ for $i < 0$, $\alpha_0 = 1$,
$$\alpha_{i+1}=\frac{(a-is)\alpha_i}{(i+1)s}$$ for $i\geq 0$ and $\alpha_{i}=0$ for $i\geq l+1$.

\item[{\rm (g)}] Let $\beta_i = a_{n-it-(u+1), -is+1}$ for $i\in \mathbb{Z}$. Then $\beta_i=0$ for $i\geq l$ and $\beta_{-u-2}=0=\beta_0$.

\item[{\rm (h)}] $s > 1$.

\item[{\rm (i)}] For $i \geq 0$, $$\beta_{i+1} = \frac{(a-is+1)\beta_i + (n-it)\alpha_i}{(i+1)s-1}.$$
\end{enumerate}
\end{Lemma}
\begin{proof}
Let $F$ be as in Lemma~\ref{l02} and let $\pi =tv$. Put $g=(F/b_m)^{\pi}=F^{\pi}/y^{\pi b}$.
Then $g$ is monic in $x$ and $\deg_x(g)=\pi m>0$. Let $g=\sum_{i=0}^n a_ix^i$, where $a_i\in k[y,1/y]$
for $0\leq i \leq n$ and $a_n\neq 0$. Then $a_n=1$, $v|n$ and $n=lt$ for some positive integer $l$.
Moreover, $d(g) = -(ax^u/y)g$, where $a=\pi b$ is a non negative integer and $v|a$. As $F^\pi=y^ag$ and $v|a$, therefore
$\psi(g)=g$. Hence $(a), (b)$, and $(c)$ follow.

From $(b)$, we have
\begin{equation}\label{eq1}
\sum_{i=0}^{n-1} (i+1) a_{i+1}yx^i+\sum_{i=u+t}^{n+u+t} (a_{i-u-t})_yy^{s+1}x^i+\sum_{i=u}^{n+u} ((a_{i-u})_yy+aa_{i-u})x^i=0.
\end{equation}

For $1\leq i \leq t-1$, equating the coefficient of $x^{n+u+i}$ in Equation~(\ref{eq1}), we get $(a_{n-t+i})_y=0$. Therefore $a_{n-t+i} \in k$.
As $\psi(g)=g$ and $v|n$, therefore $a_{n-t+i}=0$ for $1\leq i \leq t-1$. Hence $(d)$ follows.

By equating the coefficient of $x^iy^j~(0\leq i \leq n+t+u \text{ and } j\in \mathbb{Z})$ in Equation~(\ref{eq1}), we get
$$(i+1)a_{i+1,j-1} + (j-s)a_{i-t-u,j-s} + (j+a)a_{i-u,j}=0.$$
Therefore for every $p\in \mathbb{Z}$ and $q\in \mathbb{Z}$, we have
\begin{equation}\label{eq2}
-qa_{p,q} = (p+t+u+1)a_{p+t+u+1,q+s-1} + (q+s+a)a_{p+t,q+s}.
\end{equation}

We prove the statement $(e)$ by induction on $\alpha$. For $\alpha=0$, it follows from $(a)$ and $(d)$. Now, assume that
$a_{i,j} = 0$ for every pair $(i,j)$ such that either $i\geq n-\alpha t$ and $j<-\alpha s$ or $n-(\alpha+1)t < i < n-\alpha t$ and $j<-\alpha s+1$.
Then $a_{i,j}=0$ for $i>n-(\alpha +1)t$ and $j<-(\alpha +1)s$ and by Equation~(\ref{eq2}), we get $a_{n-(\alpha +1)t,j}=0$ for $j<-(\alpha +1)s$.
Again by using the induction hypothesis and Equation~(\ref{eq2}), we get $a_{i,j}=0$ for
$n-(\alpha+2)t < i < n-(\alpha +1) t$ and $j<-(\alpha +1)s+1$.
Hence $(e)$ follows.

As $n-it > n$ if $i<0$, therefore $\alpha_i=0$ if $i<0$. Clearly, $\alpha_0 = a_{n,0}=1$.
For $i \geq 0$, by $(e)$, we have $a_{n-it+u+1,-is-1}=0$, therefore from Equation~(\ref{eq2}), we get
\begin{equation}\label{eq3}
\alpha_{i+1}=\frac{(a-is)\alpha_i}{(i+1)s}.
\end{equation}
As $n=lt$, therefore $\alpha_{l+1}=a_{-t,-(l+1)s}=0$. Then by Equation~(\ref{eq3}), $\alpha_i=0$ for $i\geq l+1$.
Hence $(f)$ follows.

Clearly $\beta_{-u-2}=0=\beta_i$ for $i\geq l$ as $t\geq 1$ and $n=lt$. For $i\in \mathbb{Z}$, by Equation~(\ref{eq2}), we have
\begin{equation}\label{eq4}
((i+1)s-1)\beta_{i+1} = (a-is+1)\beta_i + (n-it)\alpha_i.
\end{equation}

As $\alpha_i=0$ for $i<0$, therefore $\beta_{-u-2}=0\Rightarrow \beta_{-u-1}=0  \Rightarrow \dots \Rightarrow \beta_0=0$.
Now, from Equation~(\ref{eq4}), we have $(s-1)\beta_1=n>0$. Therefore $s\neq 1$. Now, $(i)$ follows from Equation~(\ref{eq4}).
\end{proof}

\begin{Theorem}\label{l04}
Let $k$ be a field of characteristic zero. Let $d$ denote the $k$-derivation $\partial_x + x^u(c_1x^ty^s+c_2)\partial_y$,
where $s,t \geq 1, u\geq 0$ and $c_1, c_2\in K\setminus \{0\}$. Then $d$ is a simple derivation of $k[x,y]$.
\end{Theorem}

\begin{proof}
By Lemma~\ref{l01}, we can assume that $c_1=c_2=1$. Suppose, if possible, that $d$ is not a simple derivation
of $k[x,y]$. Then we can apply Lemma~\ref{l03}. We use the same notations as those in Lemma~\ref{l03}.
By Lemma~\ref{l03} part $(h)$, $s>1$. Let $\delta \geq 0$ be the largest integer such that $\alpha_{\delta} \neq 0$.
Then $\alpha_{\delta +1} = 0$ and $\delta \leq l$. Now, by Equation~(\ref{eq3}), we get
\begin{equation}\label{eq5}
a=\delta s,~ \alpha_i>0 \text{ for } 0\leq i\leq \delta ~\text{  and  }~ \alpha_i=0 \text{ for } i \geq \delta +1.
\end{equation}

Now as $\beta_0=0$, therefore by Equation~(\ref{eq4}) and Equation~(\ref{eq5}), we have
$$\beta_1 >0 \Rightarrow \beta_2>0 \Rightarrow \dots \Rightarrow \beta_{\delta +1}>0.$$
Therefore $\delta \neq l$ as $\beta_{l+1}=0$. Hence $\delta < l$. Now, by Equation~(\ref{eq4}) and Equation~(\ref{eq5}), we have
$\beta_l=0 \Rightarrow \beta_{l-1}=0\Rightarrow \dots \Rightarrow\beta_{\delta + 1}=0$, which is absurd. Hence $d$ is a simple
derivation of $k[x,y]$.
\end{proof}

\section{The derivation $\partial_x + x^u(c_1x^t+c_2y^s)\partial_y$ of $k[x,y]$}

For this section, we fix some notations. Let $x, y$ be indeterminates, $u\geq 0, s,t\geq 1$ be integers.
Put $v=(s-1)t+(u+1)s$. Let $\omega \in k$ be a primitive $v$th root of unity. Let $\psi$ denote the $k$-algebra automorphism of
$k[x,y]$ given by $\psi(x)=\omega x$ and $\psi(y)=\omega^{u+t+1} y$. These notations shall remain fixed throughout this section.

In this section, we discuss the simplicity of the $k$-derivation $\partial_x + x^u(c_1x^t+c_2y^s)\partial_y$ of $k[x,y]$.

We now prove some lemmas.

\begin{Lemma}\label{l11}
The $k$-derivation $\bar{d} = \partial_x + x^u(c_1x^t+c_2y^s)\partial_y$, where
$c_1, c_2\in k\setminus \{0\}$, of $k[x,y]$ is simple if and only if the $k$-derivation
$d = \partial_x + x^u(x^t+y^s)\partial_y$ of $k[x,y]$ is simple.
\end{Lemma}

\begin{proof}
Let $\alpha \in k$ be such that $\alpha^v=c_1^{1-s}c_2^{-1}$. Define the $k$-algebra automorphism
$\theta:k[x,y]\rightarrow k[x,y]$ by $\theta(x)=\alpha x$ and $\theta(y)=c_1\alpha^{u+t+1} y$.
Then $\theta \bar{d} \theta^{-1} = \alpha^{-1}d$. Therefore, $\bar{d}$ is simple if and only if $d$ is simple.
\end{proof}

\begin{Lemma}\label{l12}
Let $d$ denote the $k$-derivation $\partial_x + x^u(x^t+y^s)\partial_y$ of $k[x,y]$, where $s\geq 2$.
If $d$ is not a simple derivation of $k[x,y]$, then there exists
$F=\sum_{i=0}^m b_ix^i$, where $m>0$ and $b_i\in k[y]$ for $0\leq i \leq m$, such that
$d(F)=bx^{u}y^{s-1}F$, $b\in k$, $\psi(F)=F$ and $b_m=1$.
\end{Lemma}
\begin{proof}
By Proposition~\ref{p01}, there exist $f\in k[x,y] \setminus k$ and $\lambda \in k[x,y]$ such that $d(f)=\lambda f$. It is easy
to prove that $f\notin k[y]$. Hence $\deg_x(f) >0$. Note that if $\lambda \neq 0$, then $\deg_x(\lambda)\leq t+u$ and
$\deg_y(\lambda)\leq s-1$.

Let $F=\prod_{i=0}^{v-1} \psi^i(f)$. Then $\psi(F)=F$ and $\deg_x(F)>0$. Note that for every positive integer $i$,
we have $\psi^{-i}d\psi^i=\omega^id$ and therefore $d(F)=\Lambda F$, where $\Lambda=\sum_{i=0}^{v-1}\omega^i\psi^i(\lambda)$.
Now, if $\Lambda \neq 0$, then
$\deg_x(\Lambda)\leq \deg_x(\lambda)\leq t+u$ and $\deg_y(\Lambda)\leq \deg_y(\lambda)\leq s-1$.
As $\psi^{-1}d\psi=\omega d$ and $\psi(F)=F$, therefore $\psi(\Lambda)=\omega^{-1} \Lambda$.
For $0\leq i \leq t+u$ and $0\leq j \leq s-1$, equating the coefficient of $x^iy^j$ in $\psi(\Lambda)=\omega^{-1} \Lambda$, we
get $\Lambda = b x^{u}y^{s-1}$ for some $b\in k$.

Write $F= \sum_{i=0}^m b_ix^i$, where $m\geq 1$, $b_i\in k[y]$
for $0\leq i \leq m$ and $b_m\neq 0$. Then by equating the coefficient of $x^{m+t+u}$ in $d(F)=bx^{u}y^{s-1}F$,
we get $(b_m)_y=0$ and hence $b_m \in k\setminus \{0\}$.
Without loss of generality we may assume that $b_m =1$.
\end{proof}

\begin{Lemma}\label{l13}
Let $d$ denote the $k$-derivation $\partial_x + x^u(x^t+y^s)\partial_y$ of $k[x,y]$, where $s\geq 2$.
If $d$ is not a simple derivation of $k[x,y]$, then there exists a Darboux element
$g=\sum_{i=0}^n a_ix^i$, where $a_i\in k[y]$ for $0\leq i \leq n$, of $d$ in $k[x,y]$ such that
\begin{enumerate}
\item[{\rm (a)}] $a_n=1, n=lt$ is a positive integer, and $v|n$.

\item[{\rm (b)}] $d(g) = ax^u y^{s-1}g$, where $a\in k$.

\item[{\rm (c)}] $\psi(g) = g$.

\item[{\rm (d)}] $a_{n-t+i} = 0$ for $1\leq i \leq t-1$.

\item[{\rm (e)}] Let $a_{i,j}$ denote the coefficient of $y^j$ in $a_i$ for $0\leq i \leq n$ and $j\geq 0$;
$a_{i,j}=0$ for $0\leq i \leq n$ and $j< 0$;
and $a_{i,j}=0$ for $i\in \mathbb{Z}\setminus \{0,1,\dots,n\}$ and $j\in \mathbb{Z}$. Then for every $\alpha \geq 1$, we have
$a_{i,j} = 0$ for every $i> n-\alpha t$ and $j>\alpha s$.

\item[{\rm (f)}] Let $\alpha_i = a_{n-it, is}$ for $i\geq 0$. Then $\alpha_0 = 1$,
$$\alpha_{i+1}=\frac{(a-is)\alpha_i}{(i+1)s}$$ for $i\geq 0$ and $\alpha_{i}=0$ for $i\geq l+1$.

\item[{\rm (g)}] Let $\beta_i = a_{n-it-(u+1), (i-1)s+1}$ for $i\geq 0$. Then $\beta_i = 0$ for $i \geq l$, $\beta_0 = 0$ and for $i \geq 0$,
$$\beta_{i+1} = \frac{(a-1-(i-1)s)\beta_i - (n-it)\alpha_i}{is+1}.$$
\end{enumerate}
\end{Lemma}
\begin{proof}
Let $F$ be as in Lemma~\ref{l12} and let $\pi =tv$. Put $g=F^{\pi}$.
Then $g$ is monic in $x$ and $\deg_x(g)=\pi m>0$. Let $g=\sum_{i=0}^n a_ix^i$, where $a_i\in k[y]$
for $0\leq i \leq n$ and $a_n\neq 0$. Then $a_n=1$, $v|n$ and $n=lt$ for some positive integer $l$.
Moreover, $d(g) = ax^uy^{s-1}g$, where $a=\pi b$. As $F^\pi=g$, therefore
$\psi(g)=g$. Hence $(a), (b)$, and $(c)$ follow.

From $(b)$, we have
\begin{equation}\label{eq11}
\sum_{i=0}^{n-1} (i+1) a_{i+1}x^i + \sum_{i=u+t}^{n+u+t} (a_{i-u-t})_yx^i + \sum_{i=u}^{n+u} ((a_{i-u})_yy^s-aa_{i-u}y^{s-1})x^i =0.
\end{equation}

For $1\leq i \leq t-1$, equating the coefficient of $x^{n+u+i}$ in Equation~(\ref{eq11}), we get $(a_{n-t+i})_y=0$. Therefore $a_{n-t+i} \in k$.
As $\psi(g)=g$ and $v|n$, therefore $a_{n-t+i}=0$ for $1\leq i \leq t-1$. Hence $(d)$ follows.

By equating the coefficient of $x^iy^j~(0\leq i \leq n+t+u \text{ and } j\geq 0)$ in Equation~(\ref{eq11}), we get
$$(i+1)a_{i+1,j} + (j+1)a_{i-t-u,j+1} + (j-s+1-a)a_{i-u,j-s+1}=0.$$
Therefore for every $p, q\in \mathbb{Z}$, we have
\begin{equation}\label{eq12}
-qa_{p,q} = (p+t+u+1)a_{p+t+u+1,q-1} + (q-s-a)a_{p+t,q-s}.
\end{equation}

We prove the statement $(e)$ by induction on $\alpha$. For $\alpha=1$, it follows from $(a)$ and $(d)$. Now, assume that
$a_{i,j} = 0$ for every $i> n-\alpha t$ and $j>\alpha s$.
Then by using the induction hypothesis and Equation~(\ref{eq12}), we get $a_{i,j}=0$ for
$i > n-(\alpha +1) t$ and $j>(\alpha +1)s$.
Hence $(e)$ follows.

Clearly, $\alpha_0 = a_{n,0}=1$.
For $i \geq 0$, by $(e)$, we have $a_{n-it+u+1,(i+1)s-1}=0$, therefore from Equation~(\ref{eq12}), we get
\begin{equation}\label{eq13}
\alpha_{i+1}=\frac{(a-is)\alpha_i}{(i+1)s}.
\end{equation}
As $n=lt$, therefore $\alpha_{l+1}=a_{-t,-(l+1)s}=0$. Then by Equation~(\ref{eq13}), $\alpha_i=0$ for $i\geq l+1$.
Hence $(f)$ follows.

Clearly $\beta_i=0$ for $i\geq l$ as $n=lt$ and $\beta_0=0$ as $s\geq 2$. For $i\geq 0$, by Equation~(\ref{eq12}), we have
\begin{equation}\label{eq14}
\beta_{i+1} = \frac{(a-1-(i-1)s)\beta_i - (n-it)\alpha_i}{is+1}.
\end{equation}
\end{proof}

\begin{Theorem}\label{l24}
Let $k$ be a field of characteristic zero. Let $d$ denote the $k$-derivation $\partial_x + x^u(c_1x^t+c_2y^s)\partial_y$,
where $s\geq 2,t \geq 1, u\geq 0$ and $c_1, c_2\in K\setminus \{0\}$. Then $d$ is a simple derivation of $k[x,y]$.
\end{Theorem}

\begin{proof}
By Lemma~\ref{l11}, we can assume that $c_1=c_2=1$. Suppose, if possible, that $d$ is not a simple derivation
of $k[x,y]$. Then we can apply Lemma~\ref{l13}. We use the same notations as those in Lemma~\ref{l13}.
Let $\delta \geq 0$ be the largest integer such that $\alpha_{\delta} \neq 0$.
Then $\alpha_{\delta +1} = 0$ and $\delta \leq l$. Now, by Equation~(\ref{eq13}), we get
\begin{equation}\label{eq15}
a=\delta s,~ \alpha_i>0 \text{ for } 0\leq i\leq \delta ~\text{  and  }~ \alpha_i=0 \text{ for } i \geq \delta +1.
\end{equation}

Now as $\beta_0=0$, therefore by Equation~(\ref{eq14}) and Equation~(\ref{eq15}), we have
$$\beta_1 <0 \Rightarrow \beta_2<0 \Rightarrow \dots \Rightarrow \beta_{\delta +1}<0.$$
Therefore $\delta \neq l$ as $\beta_{l+1}=0$. Hence $\delta < l$. Now, by Equation~(\ref{eq14}) and Equation~(\ref{eq15}), we have
$\beta_l=0 \Rightarrow \beta_{l-1}=0\Rightarrow \dots \Rightarrow\beta_{\delta + 1}=0$, which is absurd. Hence $d$ is a simple
derivation of $k[x,y]$.
\end{proof}

\section{The derivation $y^r\partial_x + (c_1x^{t_1}y^{s_1}+c_2x^{t_2}y^{s_2})\partial_y$ of $k[x,y]$}

In this section, we give a necessary and sufficient condition for the $k$-derivation $y^r\partial_x + (c_1x^{t_1}y^{s_1}+c_2x^{t_2}y^{s_2})\partial_y$,
where $r, t_1, s_1, t_2, s_2 \geq 0$ and $c_1, c_2\in k$,
of $k[x,y]$ to be simple.

\begin{Theorem}
Let $k$ be a field of characteristic zero. Let $d$ denote the $k$-derivation $y^r\partial_x + (c_1x^{t_1}y^{s_1}+c_2x^{t_2}y^{s_2})\partial_y$
of $k[x,y]$, where $r,t_1, s_1, t_2, s_2 \geq 0$ are integers, $c_1, c_2 \in k$, $t_2 \leq t_1$ and $(t_1, s_1)\neq (t_2, s_2)$.
Then $d$ is a simple derivation of $k[x,y]$ if and only if the following conditions are satisfied:
\begin{enumerate}
\item[{\rm (a)}] $c_1\neq 0, c_2\neq 0$ and $t_2<t_1$.

\item[{\rm (b)}] If $r>0$, then $t_2=s_2=0$ and $s_1>0$.

\item[{\rm (c)}] If $r=0$, then $s_1+s_2>0$ and $s_1s_2=0$.

\item[{\rm (d)}] If $r=s_1=0$ and $s_2=1$, then $(t_2+1)\nmid (t_1-t_2)$.
\end{enumerate}
\end{Theorem}

\begin{proof}
Assume that $d$ is a simple derivation of $k[x,y]$. Then:

(a) Let, if possible, $c_1=0$. Then $y$ is a Darboux element of $d$ if $s_2>0$
and $(t_2+1)y^{r+1}-c_2(r+1)x^{t_2+1}$ is a Darboux element of $d$ if $s_2=0$. Which is a contradiction.
Hence $c_1\neq 0$. Similarly $c_2\neq 0$.  If $t_1=t_2$, then $s_1\neq s_2$ and hence $c_1y^{s_1}+c_2y^{s_2}$
is a Darboux element of $d$. Therefore $t_1> t_2$.

(b) Let $r>0$. If $t_2>0$, then $(x,y)$ is a $d$-differential ideal. Therefore $t_2=0$.
If $s_2>0$, then $(x,y)$ is $d$-differential. Therefore $s_2=0$.
Now, if $s_1=0$, then $(c_1x^{t_1}+c_2, y)$ is $d$-differential. Therefore $s_1>0$.

(c) Let $r=0$. If $s_1+s_2=0$, then $c_1(t_2+1)x^{t_1+1}+c_2(t_1+1)x^{t_2+1}-(t_1+1)(t_2+1)y$
is a Darboux element of $d$. Therefore $s_1+s_2>0$. If $s_1s_2\neq 0$, then $y$ is a
Darboux element of $d$. Therefore $s_1s_2=0$.

(d) Let $r=s_1=0$ and $s_2=1$. Then by Lemma~\ref{l11}, the derivation $\partial_x + (x^{t_1}+x^{t_2}y)\partial_y$
is simple. Now by Theorem~\ref{001}, $\mathfrak{p}(x^{t_2}, x^{t_1})\neq 0$ and hence
by Lemma~\ref{000}, we get $(t_2+1)\nmid (t_1-t_2)$.

Conversely, assume that the conditions (a), (b), (c) and (d) are satisfied.
By (a), the derivation $d=y^r\partial_x + x^{u}(c_1x^{t}y^{s_1}+c_2y^{s_2})\partial_y$,
where $u=t_2\geq 0$, $t=t_1-t_2\geq 1$ and $c_1, c_2\in k\setminus \{0\}$.

Case-1: Let $r>0$. Then by (b), the derivation $d=y^r\partial_x + (c_1x^{t}y^{s_1}+c_2)\partial_y$, where
$s_1, t\geq 1$. Hence by Theorem~\ref{002}, $d$ is simple.

Case-2: Let $r=0$ and $s_2=0$. Then by (c), we have $d=\partial_x + x^u(c_1x^{t}y^{s_1}+c_2)\partial_y$, where
$s_1, t\geq 1$ and $u\geq 0$. Hence by Theorem~\ref{l04}, $d$ is simple.

Case-3: Let $r=0$ and $s_2=1$. Then by (c) and (d), we have $s_1=0$ and $(t_2+1)\nmid (t_1-t_2)$. Therefore
by Lemma~\ref{000}, we have $\mathfrak{p}(x^{t_2}, x^{t_1})\neq 0$. Now, by Theorem~\ref{001},
the derivation $\partial_x + (x^{t_1}+x^{t_2}y)\partial_y=\partial_x + x^u(x^{t}+y)\partial_y$ of $k[x,y]$ is simple.
Hence by Lemma~\ref{l11}, the derivation $d=\partial_x + x^u(c_1x^{t}+c_2y)\partial_y$ is simple.

Case-4: Let $r=0$ and $s_2\geq 2$. Then by (c), we have $s_1=0$ and therefore
$d=\partial_x + x^u(c_1x^{t}+c_2y^{s_2})\partial_y$, where $u\geq 0, t\geq 1$ and $s_2\geq 2$.
Hence by Theorem~\ref{l24}, $d$ is simple.
\end{proof}

\begin{Remark}
Although the results proved in this paper are over an algebraically
closed field of characteristic zero, they remain valid over an arbitrary field of
characteristic zero by \cite[Proposition~13.1.3]{n94}.
\end{Remark}

\end{document}